\documentclass[11pt]{article}
\usepackage{amsmath,amssymb,amsthm,fullpage}
\usepackage{graphicx}
\usepackage{hyperref}

\theoremstyle{plain}
\newtheorem{theorem}{Theorem}
\newtheorem{proposition}{Proposition}
\newtheorem{lemma}{Lemma}
\newtheorem{corollary}{Corollary}
\theoremstyle{definition}
\newtheorem{definition}{Definition}
\theoremstyle{remark}
\newtheorem{remark}{Remark}

\newcommand{\bbN}{\mathbb{N}}
\newcommand{\bbZ}{\mathbb{Z}}

\title{Resolution of Erd\H{o}s Problem \#728: a writeup of Aristotle's Lean proof}
\author{Nat Sothanaphan\footnote{natsothanaphan@gmail.com}}
\date{Jan 27, 2026}

\begin{document}
\maketitle

\begin{abstract}
We provide a writeup of a resolution of Erd\H{o}s Problem \#728; this is the first Erd\H{o}s problem (a problem proposed by Paul Erd\H{o}s which has been collected in the Erd\H{o}s Problems website \cite{Bloom728}) regarded as fully resolved autonomously by an AI system. The system in question is a combination of GPT-5.2 Pro by OpenAI and Aristotle by Harmonic, operated by Kevin Barreto. The final result of the system is a formal proof written in Lean, which we translate to informal mathematics in the present writeup for wider accessibility.

The proved result is as follows. We show a logarithmic-gap phenomenon regarding factorial divisibility: For any constants $0<C_1<C_2$ and $0 < \varepsilon < 1/2$ there exist infinitely many triples $(a,b,n)\in\mathbb N^3$ with $\varepsilon n \le a,b \le (1-\varepsilon)n$ such that
\[
a!\,b!\mid n!\,(a+b-n)!\qquad\text{and}\qquad
C_1\log n < a+b-n < C_2\log n.
\]
The argument reduces this to a binomial divisibility
$\binom{m+k}{k}\mid\binom{2m}{m}$ and studies it prime-by-prime. By Kummer's theorem, $\nu_p\binom{2m}{m}$ translates into a carry count for doubling $m$ in base $p$.
We then employ a counting argument to find, in each scale $[M,2M]$, an integer $m$ whose base-$p$ expansions simultaneously force many carries when doubling $m$, for every prime $p\le 2k$, while avoiding the rare event that one of $m+1,\dots,m+k$ is divisible by an unusually high power of $p$.
These ``carry-rich but spike-free'' choices of $m$ force the needed $p$-adic inequalities and the divisibility. The overall strategy is similar to results regarding divisors of $\binom{2n}{n}$ studied earlier by Erd\H{o}s \cite{Erdos557} and by Pomerance \cite{Pomerance2015}.
\end{abstract}

\section{Introduction}
Erd\H{o}s Problem \#728 \cite{Bloom728} asks how large the gap
\[
k:=a+b-n
\]
can be, assuming the factorial divisibility
\[
a!\,b!\ \mid\ n!\,k!.
\]
Equivalently, writing $N=a+b$, this asks for large $k$ such that
\[
\binom{N}{k}\ \Big|\ \binom{N}{a}.
\]
Without additional restrictions there are easy/extremal families (for instance, one may take $k$ huge by taking $a$ and $b$ very large).
One way to exclude such examples is to impose a condition $a,b \le (1-\varepsilon) n$ for some $\varepsilon > 0$; our construction actually satisfies a much stronger criterion,
with $a$ and $b$ extremely close to $n/2$, but Theorem \ref{thm:window} below only records the size of~$k$.

\medskip
The qualitative feature exploited in the proof is that the right-hand side contains the full factorial $k!$.
For a fixed prime $p$, this contributes the term $\nu_p(k!)$, which grows like $k/(p-1)$ and thus can be arranged to dominate sporadic
$p$-adic contributions coming from the factor $(m+i)$ in a consecutive block.
After an explicit change of variables, the problem reduces to proving the binomial divisibility
\[
\binom{m+k}{k}\ \Big|\ \binom{2m}{m},
\]
where we will take $k\asymp \log m$.
On the $p$-adic level this becomes the inequality
\[
V_p(m,k)\ \le\ \kappa_p(m)\qquad\text{for every prime $p$,}
\]
where $V_p(m,k):=\max_{1\le i\le k}\nu_p(m+i)$ and $\kappa_p(m):=\nu_p\binom{2m}{m}$.
Kummer's theorem identifies $\kappa_p(m)$ with the number of carries when adding $m+m$ in base~$p$.
Thus the task is: choose $m$ so that doubling in base $p$ produces many carries, for all relevant primes $p$,
while avoiding the rare event that one of $m+1,\dots,m+k$ is divisible by an unusually large power of~$p$.

\begin{figure}[h]
\centering
\includegraphics[width=0.8\linewidth]{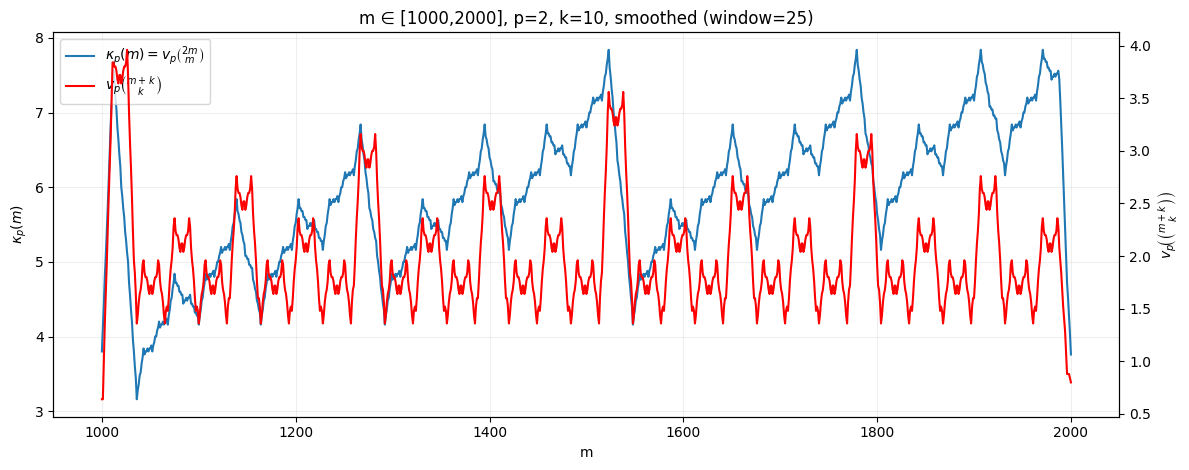}
\includegraphics[width=0.8\linewidth]{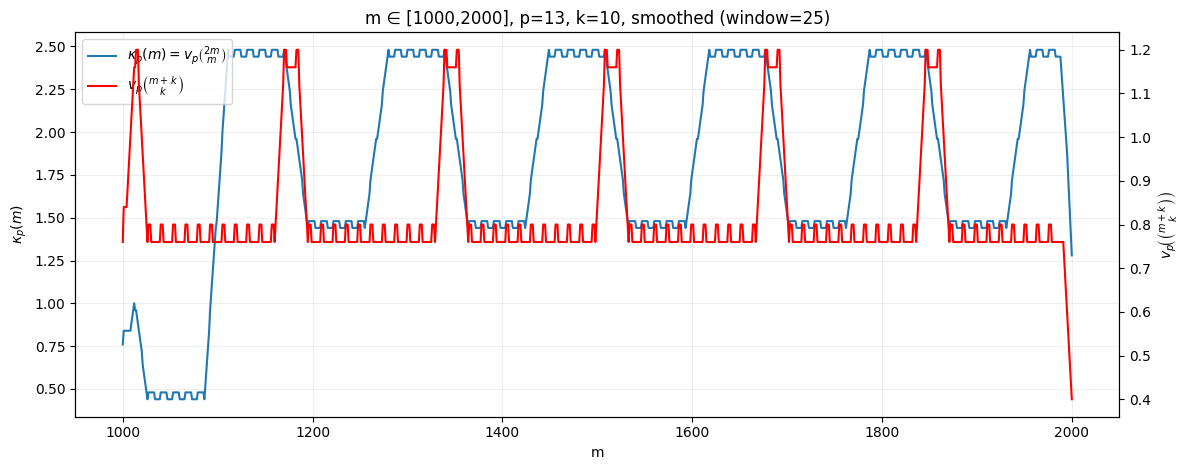}
\caption{Plots showing the required inequality $\nu_p(\binom{m+k}{k}) \le \kappa_p(m)$; we want to find places where the red is below the blue. Both plots have $m \in [1000, 2000]$ and $k=10$. The two plots are for primes $p=2$ and $p=13$, respectively. Note that these plots are smoothed by a simple moving average with window $25$, since the unsmoothed plots are dense and difficult to read. Plotting code was partially provided by ChatGPT.}
\end{figure}

We split primes into two regimes.
If $p>2k$ then any divisibility $p^J\mid (m+i)$ with $1\le i\le k$ forces $J$ carries, so this range is handled quickly.
The main work is for primes $p\le 2k$.
There we force many carries by demanding that many of the first $L_p$ base-$p$ digits of $m$ are at least $\lceil p/2\rceil$ and exclude the exceptional set where some $m+i$ is divisible by an unusually high power $p^{J_p+t}$.
A counting argument on $m\in[M,2M]$ shows that for every large $M$ there exists an $m$ satisfying these conditions for all primes $p\le 2k$.
The fact that $k\asymp \log M$ then yields the desired logarithmic gap.

\begin{theorem}[Logarithmic gap window]\label{thm:window}
Fix $0<C_1<C_2$. Let $0 < \varepsilon < 1/2$. There exist infinitely many $(a,b,n)\in\mathbb N^3$ with $\varepsilon n \le a,b \le (1-\varepsilon)n$ such that 
\[
a!\,b!\mid n!\,(a+b-n)!\qquad\text{and}\qquad C_1\log n < a+b-n < C_2\log n.
\]
\end{theorem}

\begin{remark}
Terence Tao observed in \cite{Bloom728} that this result is not in its best possible form. Indeed, the main constraint is $k \cdot e^{-\tilde{\mu}(M)/8} = o(1)$ step in Lemma \ref{lem:bad-small}, which can be satisfied up to $k = \exp(c \sqrt{\log n})$ for some small constant $c > 0$.
\end{remark}

After an earlier version of this writeup was completed, Carl Pomerance has written up a note \cite{Pomerance2026} extending his previous work \cite{Pomerance2015} to arrive at results similar to ours. Independently, Erd\H{o}s problems \#729 \cite{Bloom729} and \#401 \cite{Bloom401} were also solved by modifications of the argument in this writeup. We discuss these further developments in the second Appendix.

\medskip
\textbf{Acknowledgements}. The author thanks forum participants in \cite{Bloom728} for their fruitful discussion. They are, in alphabetical order, Boris Alexeev, Kevin Barreto, Thomas Bloom, Moritz Firsching, username KoishiChan, username Leeham, username old-bielefelder, Terence Tao. Specific to producing this writeup: Boris Alexeev ran Aristotle to simplify the proof; Kevin Barreto produced the original proofs; KoishiChan performed literature search; Terence Tao suggested further ideas and found extra literature. Finally, the author thanks Carl Pomerance for extending his work in light of the present developments.

\section{Related literature and provenance}

Erd\H{o}s Problem \#728, as stated in \cite{Bloom728}, asks how large the gap $a+b-n$ can be under the factorial divisibility
\[
a!\,b!\ \mid\ n!\,(a+b-n)!.
\]
This problem originally appears in \cite{EGRS1975}.

\medskip
The proof strategy is closely related to several divisibility results for central binomial coefficients, usually governed by Kummer's carry interpretation of $p$-adic valuations.

Erd\H{o}s' problem \emph{Aufgabe~557} in Elemente der Mathematik \cite{Erdos557} concerns the stronger divisibility $a!\,b!\mid n!$
(i.e.\ the case $k=0$).
Erd\H{o}s proved an upper bound $a+b<n+O(\log n)$ in this setting, and observed that this is sharp up to the constant.
The published solutions show a strong ``almost all $n$'' statement:
for a sufficiently small absolute constant $C>0$ and $a=\lfloor C\log n\rfloor$,
one has for almost all $n$ the interval product divisibility
\[
(n+1)\cdots(n+a)\ \Big|\ \binom{2n}{n}.
\]
This is close in spirit to the present proof.
Indeed, after rewriting $\binom{m+k}{k}=(m+1)\cdots(m+k)/k!$,
our target divisibility $\binom{m+k}{k}\mid\binom{2m}{m}$ is a ``binomial-coefficient'' analogue of such interval-product divisibilities.

The paper of Erd\H{o}s--Graham--Ruzsa--Straus \cite{EGRS1975} which contains the present question itself studies the distribution of prime factors of $\binom{2n}{n}$,
quantifying, among other things, how small primes contribute high powers and how primes are avoided. We do not use specific results from \cite{EGRS1975}.

Pomerance \cite{Pomerance2015} studies questions about divisors of the central binomial coefficient $\binom{2n}{n}$.
One of his main theorems shows that for each fixed $k\ge 1$, the set of $n$ for which $n+k$ divides $\binom{2n}{n}$ has asymptotic density~$1$
(\cite[Theorem~2]{Pomerance2015}).
The methods in \cite{Pomerance2015} exploit Kummer's theorem and the idea that base-$p$ digit patterns behave ``randomly.''
Our argument is in the same vein, but differs in two key respects:
(i) we treat a \emph{growing} window length $k\asymp \log n$ rather than a fixed $k$,
and (ii) we work with a structured divisor $\binom{m+k}{k}$ of $\binom{2m}{m}$ rather than a single linear factor $n+k$.

After Pomerance's paper, there has been further progress on related divisibility questions for $\binom{2n}{n}$, in several directions.
Ford and Konyagin \cite{FordKonyagin2021} prove that for each fixed $\ell\in\mathbb N$ the set of $n$ such that $n^\ell\mid \binom{2n}{n}$ has a positive asymptotic density $c_\ell$, and they obtain an asymptotic formula for $c_\ell$ as $\ell\to\infty$.
They also show that $\#\{n\le x:(n,\binom{2n}{n})=1\}\sim c x/\log x$ for an explicit constant $c$.

Croot, Mousavi, and Schmidt \cite{CrootMousaviSchmidt2024} study a problem motivated by Graham's conjecture on \\ $\gcd\!\bigl(\binom{2n}{n},105\bigr)$.
They show that for any fixed $r\ge 1$ and any $0<\varepsilon<1/(20r^2)$, there is a threshold $p_0(r,\varepsilon)$ such that for any distinct primes $p_1,\dots,p_r\ge p_0(r,\varepsilon)$ there exist infinitely many $n$ for which
\[
\nu_{p_j}\!\left(\binom{2n}{n}\right)\ \le\ \varepsilon\,\frac{\log n}{\log p_j}\qquad (j=1,\dots,r).
\]
By Kummer's theorem, $\nu_{p}\!\bigl(\binom{2n}{n}\bigr)$ equals the number of carries when adding $n+n$ in base $p$.
Thus \cite{CrootMousaviSchmidt2024} constructs integers $n$ that are \emph{carry-poor} for doubling in each of the bases $p_1,\dots,p_r$ simultaneously.

Bloom and Croot \cite{BloomCroot2025} strengthen this line of work by studying integers with small digits in multiple bases.
For distinct coprime bases $g_1,\dots,g_r$ that are sufficiently large, depending on $r$, and any $\varepsilon>0$, they produce infinitely many $n$ for which all but $\varepsilon\log n$ base-$g_i$ digits are $\le g_i/2$ simultaneously for all $i$. This establishes a ``weak'' version of the above Graham's conjecture.

The work of Croot--Mousavi--Schmidt and Bloom--Croot is primarily ``carry-poor'': it produces integers whose base-$p$ expansions are arranged so that doubling creates few carries in several fixed bases.
In contrast, the present proof is ``carry-rich'', forcing many carries simultaneously.

\section{Reduction to a valuation inequality}\label{sec:reduction}
Let $m\ge 1$ and $k\ge 1$ be integers and set
\begin{equation}\label{eq:family}
n:=2m,\qquad b:=m,\qquad a:=m+k.
\end{equation}
Then $b=n/2$ and $a+b-n=k$.

We will choose a large scale $M$ and search for some $m\in[M,2M]$.
We set
\begin{equation}\label{eq:kval}
k := \lfloor c\log M\rfloor,
\end{equation}
where $c>0$ is a constant. In the final step (proof of Theorem~\ref{thm:window}) we choose any $c\in(C_1,C_2)$ to obtain the two-sided logarithmic window.

Thus, the main work is the factorial divisibility in Theorem~\ref{thm:window}. With $n=2m$, $b=m$, $a=m+k$, the factorial divisibility becomes
\begin{equation}\label{eq:target-factorial}
(m+k)!\,m!\ \mid\ (2m)!\,k!.
\end{equation}

\begin{lemma}[Binomial reformulation]\label{lem:binom-reform}
For all $m,k\in\bbN$,
\[
(m+k)!\,m!\ \mid\ (2m)!\,k!
\quad\Longleftrightarrow\quad
\binom{m+k}{k}\ \Big|\ \binom{2m}{m}.
\]
\end{lemma}
\begin{proof}
Immediate.
\end{proof}

Fix a prime $p$. Define
\[
\kappa_p(m)\ :=\ \nu_p\!\left(\binom{2m}{m}\right),
\qquad
W_p(m,k)\ :=\ \nu_p\!\Big(\mathbb{P}od_{i=1}^k (m+i)\Big),
\qquad
V_p(m,k)\ :=\ \max_{1\le i\le k}\ \nu_p(m+i).
\]
Because
\(
\binom{m+k}{k}=\dfrac{\mathbb{P}od_{i=1}^k (m+i)}{k!},
\)
we have the exact identity
\begin{equation}\label{eq:val-reduction}
\nu_p\!\left(\binom{m+k}{k}\right)\ =\ W_p(m,k)-\nu_p(k!).
\end{equation}

\begin{lemma}[Valuation reduction]\label{lem:val-reduction}
For all primes $p$ and all $m,k$,
\[
\nu_p\!\left(\binom{m+k}{k}\right)\le \kappa_p(m)
\quad\Longleftrightarrow\quad
W_p(m,k)\le \kappa_p(m)+\nu_p(k!).
\]
\end{lemma}
\begin{proof}
Immediate from \eqref{eq:val-reduction}.
\end{proof}

Thus, it suffices to show $W_p(m,k)\le \kappa_p(m)+\nu_p(k!)$ for every prime $p$.

\begin{remark}
The dominant contribution to $W_p(m,k)$ for small primes is typically $\sim k/(p-1)$,
but \eqref{eq:val-reduction} subtracts $\nu_p(k!)$, which has the \emph{same main term}. This makes it easier to establish the required inequality.
\end{remark}

\begin{lemma}[Kummer's theorem]\label{lem:kummer}
For a prime $p$, $\kappa_p(m)$ equals the number of carries when adding $m+m$ in base $p$.
\end{lemma}

We now obtain an intermediate bound which is useful in reducing the problem to a simpler inequality.

\begin{lemma}[Interval valuation bound]\label{lem:W-upper}
For every prime $p$ and all $m,k$,
\[
W_p(m,k)\ \le\ \nu_p(k!) + V_p(m,k).
\]
\end{lemma}

\begin{proof}
For $j\ge 1$ let
\[
N_j := \#\{1\le i\le k:\ p^j\mid (m+i)\}.
\]
Then
\[
W_p(m,k)=\sum_{i=1}^k \nu_p(m+i)=\sum_{j\ge 1} N_j,
\]
since each term $(m+i)$ contributes $1$ to $N_j$.

Among $k$ consecutive integers, the number of multiples of $p^j$ is at most $\lceil k/p^j\rceil$, so
\[
N_j \le \left\lceil \frac{k}{p^j}\right\rceil\qquad (j\ge 1).
\]

Let $J:=\lfloor \log_p k\rfloor$, so $p^J\le k<p^{J+1}$. For $1\le j\le J$ we have
$\lceil k/p^j\rceil\le \lfloor k/p^j\rfloor+1$. For $j\ge J+1$ we have $p^j>k$, hence among $m+1,\dots,m+k$ there is at most one multiple of $p^j$, so $N_j\le 1$. Finally $N_j=0$ for all $j>V_p(m,k)$. Therefore
\[
W_p(m,k)=\sum_{j\ge 1}N_j
\le \sum_{j=1}^{J}\left(\left\lfloor \frac{k}{p^j}\right\rfloor+1\right) + \sum_{j=J+1}^{V_p(m,k)} 1
= \sum_{j=1}^{J}\left\lfloor \frac{k}{p^j}\right\rfloor + V_p(m,k).
\]
By Legendre's formula, $\nu_p(k!) = \sum_{j=1}^{J}\lfloor k/p^j\rfloor$. This establishes the claim.
\end{proof}

Combining Lemma~\ref{lem:val-reduction} with Lemma~\ref{lem:W-upper}, it is enough to show
\[
V_p(m,k)\ \le\ \kappa_p(m)\qquad\text{for every prime $p$}.
\]

\section{Prime-by-prime analysis via carries}\label{sec:analysis}

We now analyze the inequality $V_p(m,k)\le \kappa_p(m)$ in ranges $p>2k$ and $p\le 2k$.

\subsection{The range $p>2k$}
In this range, the desired inequality holds for free.

\begin{lemma}[Large prime lemma]\label{lem:large-primes}
If $p$ is prime and $p>2k$, then for all $m$,
\[
\kappa_p(m)\ \ge\ W_p(m,k) = V_p(m,k).
\]
\end{lemma}

\begin{proof}
Because $p>k$, at most one of the $k$ integers $m+1,\dots,m+k$ can be divisible by $p$;
more generally, for each $j\ge 1$, at most one of $m+1,\dots,m+k$ can be divisible by $p^j$. Consequently, $W_p(m,k)=V_p(m,k)$.

If $V_p(m,k)=0$ then the claim is trivial. Otherwise let $J:=V_p(m,k)\ge 1$
and pick $i\in\{1,\dots,k\}$ such that $p^J\mid (m+i)$.
Write $m+i=p^J u$ with $p\nmid u$.
Since $i\le k < p/2$, in base $p$, the lowest $J$ digits of $m$
agree with those of $p^J-i$.
But $p^J-i$ has base-$p$ expansion with the lowest digit equal to $p-i$ and the next $J-1$ digits equal to $p-1$.
All these digits are $\ge (p+1)/2$. In the addition $m+m$ in base $p$, a digit $a\ge (p+1)/2$ forces a carry at that position regardless of incoming carry. Therefore the lowest $J$ digit positions contribute at least $J$ carries. By Kummer's theorem (Lemma~\ref{lem:kummer}), $\kappa_p(m)$ equals the total number of carries, so $\kappa_p(m)\ge J$.
\end{proof}

\subsection{The range $p\le 2k$}

We will enforce a carry lower bound by inspecting the first few base-$p$ digits of $m$.
To make a counting argument on $m\in[M,2M]$ work cleanly, we choose a digit depth $L_p$ so that $p^{L_p}$ is slightly smaller than $M$. Concretely, set
\[
\eta:=\frac{1}{10},\qquad
L_p:=\left\lfloor \frac{(1-\eta)\log M}{\log p}\right\rfloor.
\]
Among residues modulo $p^{L_p}$, the base-$p$ digits are uniform in $\{0,1,\dots,p-1\}$.
A digit is at least $\lceil p/2\rceil$ with probability
\[
\theta(p):=
\begin{cases}
\frac12,& p=2,\\[2pt]
\frac{p-1}{2p},& p\ge3,
\end{cases}
\]
so the expected number of large digits among the first $L_p$ digits is
\[
\mu_p:=L_p\,\theta(p).
\]

\subsubsection{Carry lower bound and spike control}

Let $X_p(m)$ be the number of the first $L_p$ base-$p$ digits of $m$ which are $\ge \lceil p/2\rceil$.

\begin{lemma}[Forced carries from large digits]\label{lem:forced-carries-small}
For every prime $p$ and every $m$,
\[
\kappa_p(m)\ \ge\ X_p(m).
\]
\end{lemma}
\begin{proof}
Write $m$ in base $p$ as $m=\sum_{j\ge 0} a_j p^j$ with $0\le a_j<p$.
When adding $m+m$ in base $p$, the digit at position $j$ is obtained from $a_j+a_j$ plus an incoming carry $c_{j-1}\in\{0,1\}$.
If $a_j\ge \lceil p/2\rceil$ then a carry must occur at position $j$.
Therefore, each of the $X_p(m)$ digit positions counted by $X_p(m)$ produces a carry.
\end{proof}

Recall $V_p(m,k)=\max_{1\le i\le k}\nu_p(m+i)$. Typically $V_p(m,k)$ is close to $\log_p k$, but it can be larger if one of the integers $m+i$
happens to be divisible by a high power $p^J$.
Such events are rare because they force $m$ into a single residue class modulo $p^J$.
We set
\[
J_p:=\lfloor \log_p k\rfloor,\qquad t(M):=\left\lceil 10\log\log M\right\rceil,
\]
and we will require the \emph{no-spike} condition $V_p(m,k)<J_p+t(M)$.
On the other hand, we will also require the \emph{carry} condition $X_p(m)\ge \mu_p/2$.

The following lemma relates these conditions.

\begin{lemma}[Threshold inequality]\label{lem:weak-threshold}
There exists $M_0$ such that for all $M\ge M_0$ and all primes $p\le 2k$,
\[
\frac{\mu_p}{2}\ \ge\ J_p+t(M)+3.
\]
\end{lemma}

\begin{proof}
We have $\theta(p)\ge 1/3$ for all primes. Thus
\[
\frac{\mu_p}{2}\ge \frac{L_p}{6} \ge \frac{1}{6}\left(\frac{(1-\eta)\log M}{\log p}-1\right).
\]
Since $p\le 2k\le 2c\log M$,
\[
\frac{\mu_p}{2}\ \ge\ \frac{(1-\eta)\log M}{6\log(2c\log M)} - \frac16.
\]
The right-hand side is $\asymp \log M/\log\log M$,
while $J_p\le \log_p k \ll \log\log M$ and $t(M)\ll \log\log M$.
Hence the inequality holds for all large $M$.
\end{proof}

Now assume $M$ is large enough that Lemma~\ref{lem:weak-threshold} holds.
Suppose $m\in[M,2M]$ satisfies the two conditions
\[
X_p(m)\ \ge\ \frac{\mu_p}{2}
\qquad\text{and}\qquad
V_p(m,k)\ <\ J_p+t(M).
\]
By Lemma~\ref{lem:forced-carries-small}, we have a chain of inequality resulting in $V_p(m,k)\le \kappa_p(m)$, as desired.

Thus, we need to find such $m$, which we will call \emph{good for $p$}.
In the next part we show that for large $M$ there exists an $m\in[M,2M]$ that is good for all primes $p\le 2k$ simultaneously.

\subsubsection{Existence of a good $m$}\label{sec:count}

This will be a probabilistic argument. We first define ``bad'' events to avoid.

\begin{definition}[Bad carry]\label{def:badcarry}
For a prime $p\le 2k$, define $\mathrm{BadCarry}_p(M)$ to be the set of $m\in[M,2M]\cap\bbN$ such that
\[
X_p(m) < \mu_p/2.
\]
\end{definition}

\begin{definition}[Bad spike]\label{def:badspike}
For a prime $p\le 2k$, define $\mathrm{BadSpike}_p(M)$ to be the set of $m\in[M,2M]\cap\bbN$ such that
\[
V_p(m,k)\ge J_p+t(M).
\]
\end{definition}

Let $\mathrm{Bad}(M)$ be the union of these sets:
\[
\mathrm{Bad}(M):=\bigcup_{p\le 2k\ \mathrm{prime}}\big(\mathrm{BadCarry}_p(M)\cup \mathrm{BadSpike}_p(M)\big).
\]
The target is to show $|\mathrm{Bad}(M)|<M+1$ for large $M$.

Both bad events can be described as unions of residue classes modulo a power of $p$:
for carries the relevant modulus is $p^{L_p}$, while for spikes it is $p^{J_p+t(M)}$.

The key building block is the following lemma.

\begin{lemma}[Residue-class counting]\label{lem:residue-uniform}
Let $Q\ge 1$ and $S\subset \bbZ/Q\bbZ$. Then
\[
\#\{m\in[M,2M]\cap\bbN:\ m\bmod Q\in S\}
\le
|S|\left(\frac{M+1}{Q}+2\right).
\]
\end{lemma}

\begin{proof}
Each residue class modulo $Q$ occurs at most $\lceil (M+1)/Q\rceil+1\le (M+1)/Q+2$ times in the interval.
Summing over the $|S|$ classes gives the bound.
\end{proof}

\begin{lemma}[Uniformity consequence]\label{lem:mod-uniform}
Assume $Q\le M^{1-\eta}$ and let $S\subset \bbZ/Q\bbZ$.
Then
\[
\frac{1}{M+1}\#\{m\in[M,2M]\cap\bbN:\ m\bmod Q\in S\}
\ \le\ \frac{|S|}{Q} + \frac{2}{M^{\eta}}.
\]
\end{lemma}
\begin{proof}
Divide the bound of Lemma~\ref{lem:residue-uniform} by $M+1$ and use $|S|\le Q$.
\end{proof}

\begin{remark}
In this writeup we never use Lemma \ref{lem:mod-uniform}; we use Lemma \ref{lem:residue-uniform} directly. But it is recorded here because the Lean file \cite{AlexeevLean} utilizes it.
\end{remark}

The event $m\in\mathrm{BadCarry}_p(M)$ depends only on the residue of $m$ modulo $p^{L_p}$.
Write residues $r\bmod p^{L_p}$ in base $p$:
$r=\sum_{j=0}^{L_p-1} a_j p^j$ with $0\le a_j<p$.
Let $\xi_j(r)=1$ iff $a_j\ge \lceil p/2\rceil$, and set $X(r):=\sum_{j=0}^{L_p-1}\xi_j(r)$.
We have that $X(r)$ is exactly a binomial variable
$\mathrm{Bin}(L_p,\theta(p))$ with mean $\mu_p=L_p\theta(p)$.

\begin{lemma}[Chernoff lower tail]\label{lem:chernoff}
If $X\sim \mathrm{Bin}(L,q)$ has mean $\mu=Lq$, then
\[
\mathbb{P}(X\le \mu/2)\ \le\ e^{-\mu/8}.
\]
\end{lemma}

\begin{proof}
This is the standard Chernoff inequality $\mathbb{P}(X\le (1-\delta)\mu)\le e^{-\delta^2\mu/2}$ with $\delta=1/2$.
\end{proof}

Let $\mathcal{R}_p(M)\subset \bbZ/p^{L_p}\bbZ$ be the set of residues with $X(r)\le \mu_p/2$.
Lemma~\ref{lem:chernoff} gives $|\mathcal{R}_p(M)|\le p^{L_p}e^{-\mu_p/8}$.
Since $\mathrm{BadCarry}_p(M)$ is exactly the set of $m$ whose residue mod $p^{L_p}$ lies in $\mathcal{R}_p(M)$, we have the following bound.

\begin{lemma}[Bad-carry count]\label{lem:badcarry-count}
For each prime $p\le 2k$,
\[
|\mathrm{BadCarry}_p(M)|
\ \le\
(M+1)e^{-\mu_p/8}+2p^{L_p}.
\]
\end{lemma}

\begin{proof}
Apply Lemma~\ref{lem:residue-uniform} with $Q=p^{L_p}$ and $S=\mathcal{R}_p(M)$:

\[
|\mathrm{BadCarry}_p(M)|
\le p^{L_p}e^{-\mu_p/8}\left(\frac{M+1}{p^{L_p}}+2\right)
\le (M+1)e^{-\mu_p/8}+2p^{L_p}.
\]

\end{proof}

We now consider the spike condition. If $V_p(m,k)$ is large, then for some $i\le k$ the integer $m+i$ is divisible by a high power of $p$.

\begin{lemma}[Bad-spike count]\label{lem:badspike-count}
For each prime $p\le 2k$,
\[
|\mathrm{BadSpike}_p(M)|
\ \le\
k\left(\frac{M+1}{p^{J_p+t(M)}}+2\right).
\]
In particular, since $p^{J_p}\le k<p^{J_p+1}$,
\[
|\mathrm{BadSpike}_p(M)|\ \le\ (M+1)\,p^{1-t(M)} + 2k.
\]
\end{lemma}

\begin{proof}
If $V_p(m,k)\ge J_p+t(M)$, then there exists $i\in\{1,\dots,k\}$ with $p^{J_p+t(M)}\mid(m+i)$, i.e.
$m\equiv -i\pmod{p^{J_p+t(M)}}$.
For a fixed $i$, Lemma~\ref{lem:residue-uniform} with $Q=p^{J_p+t(M)}$ and a singleton set $S=\{-i\}$ gives at most
$(M+1)/p^{J_p+t(M)}+2$ solutions in $[M,2M]$.
Summing over the $k$ choices of $i$ gives the first bound. The second bound is immediate.
\end{proof}

Now we show existence of a good $m$ via union bound.

\begin{lemma}[Bad set size]\label{lem:bad-small}
There exists $M_0$ such that for all $M\ge M_0$ one has $|\mathrm{Bad}(M)|<M+1$.
\end{lemma}

\begin{proof}
We sum the bounds from Lemmas~\ref{lem:badcarry-count} and \ref{lem:badspike-count} over primes $p\le 2k$. Note that
\[
p^{L_p} \le M^{1-\eta}.
\]

For carry failures:
\[
\sum_{p\le 2k}|\mathrm{BadCarry}_p(M)|
\le (M+1)\sum_{p\le 2k}e^{-\mu_p/8} + 2\sum_{p\le 2k}p^{L_p}
\le (M+1)\cdot 2k\cdot \max_{p\le 2k} e^{-\mu_p/8} + 4k\,M^{1-\eta}.
\]
As in the proof of Lemma \ref{lem:weak-threshold},
\[
\mu_p \ge\ \frac{1}{3}\left(\frac{(1-\eta)\log M}{\log(2c\log M)}-1\right)=:\tilde{\mu}(M),
\]
uniformly for all $p\le 2k$.
Since $\tilde{\mu}(M)\asymp \log M/\log\log M$, we get $k\cdot e^{-\tilde{\mu}(M)/8}=o(1)$.
Thus the first term is $o(M)$. 
The second term is also easily $o(M)$.
So for large $M$ the carry sum is $<(M+1)/3$.

For spike failures:
\[
\sum_{p\le 2k}|\mathrm{BadSpike}_p(M)|
\le (M+1)\sum_{p\le 2k} p^{1-t(M)} + 2k\cdot \pi(2k)
\le (M+1)(2k)\cdot 2^{1-t(M)} + 4k^2,
\]
where $\pi(\cdot)$ is the prime counting function.
Because $t(M)=\lceil 10\log\log M\rceil$, the first term is $o(M)$.
Moreover, $k^2=o(M)$.
Thus for large $M$ the spike sum is also $<(M+1)/3$.

Therefore $|\mathrm{Bad}(M)|<M+1$ for all sufficiently large $M$.
\end{proof}

\begin{lemma}[Existence of a good $m$]\label{lem:good-m}
For all sufficiently large $M$, there exists $m\in[M,2M]\cap\bbN$ such that for every prime $p\le 2k$,
\[
X_p(m)\ge \mu_p/2
\quad\text{and}\quad
V_p(m,k)<J_p+t(M).
\]
\end{lemma}

\begin{proof}
The interval $[M,2M]\cap\bbN$ has size $M+1$ while $\mathrm{Bad}(M)$ has size $<M+1$ by Lemma~\ref{lem:bad-small}.
\end{proof}

\section{Completion of the proof}\label{sec:finish}
We now complete the argument and prove the logarithmic window result.

\begin{proof}[Proof of Theorem~\ref{thm:window}]
Choose any constant $c$ with $C_1<c<C_2$.
For each sufficiently large $M$, set $k=\lfloor c\log M\rfloor$ and apply Lemma~\ref{lem:good-m}
to obtain an $m\in[M,2M]$ satisfying the small-prime carry and no-spike conditions for all primes $p\le 2k$.

Fix such an $m$.
For primes $p\le 2k$, the good-$m$ conditions imply $V_p(m,k)\le \kappa_p(m)$.
For primes $p>2k$, Lemma~\ref{lem:large-primes} implies $V_p(m,k)\le \kappa_p(m)$.
Therefore, using Lemma~\ref{lem:W-upper}, Lemma~\ref{lem:val-reduction}, and Lemma~\ref{lem:binom-reform}, this implies the factorial divisibility
$a!\,b!\mid n!\,(a+b-n)!$ for the triple
\[
n:=2m,\qquad b:=m,\qquad a:=m+k.
\]

It remains to verify the window bounds. Recall $k=a+b-n$.
Since $m\in[M,2M]$, $\log n = \log M+O(1)$ as $M\to\infty$.
So $k/\log n\to c$ as $M\to\infty$. Hence for large $M$,
\[
C_1\log n < k < C_2\log n.
\]
Letting $M\to\infty$ along an infinite sequence yields infinitely many triples.
\end{proof}

\section{Correspondence with the Lean development}
The accompanying Lean file \texttt{Erdos728b.lean} \cite{AlexeevLean} contains a fully formal proof of the main theorem. We map our lemmas to the Lean file as follows.

\begin{center}
\begin{tabular}{p{0.20\linewidth} p{0.73\linewidth}}
\hline
\textbf{In this writeup} & \textbf{In \texttt{Erdos728b.lean}}\\
\hline
Lemma~\ref{lem:binom-reform} &
Handled by rewriting the factorial divisibility goal in \texttt{good\_triples} and in the proof of \texttt{erdos\_728\_fc} \\
Lemma~\ref{lem:val-reduction} &
\texttt{lemma\_val\_reduction}\\
Lemma~\ref{lem:kummer} &
Mathlib lemma \texttt{padicValNat\_choose}, used in particular in \texttt{lemma\_forced\_carries\_largep} and \texttt{lemma\_forced\_carries\_smallp}\\
Lemma~\ref{lem:W-upper} &
Proved inside \texttt{lemma\_small\_primes\_good}, using \texttt{lemma\_W\_eq\_sum\_N\_pj} and \texttt{lemma\_N\_pj\_le\_ceil} \\
Lemma~\ref{lem:large-primes} &
\texttt{lemma\_p\_gt\_2k}, which invokes \texttt{lemma\_forced\_carries\_largep} \\
Lemma~\ref{lem:forced-carries-small} &
\texttt{lemma\_forced\_carries\_smallp}\\
Lemma~\ref{lem:weak-threshold} &
\texttt{lemma\_threshold\_weak\_uniform}, building on \texttt{lemma\_threshold\_reduction} and auxiliary asymptotic bounds \\
Lemma~\ref{lem:residue-uniform} &
\texttt{lemma\_residue\_interval}\\
Lemma~\ref{lem:mod-uniform} &
\texttt{lemma\_mod\_uniform}, with the size condition $Q_p\le M^{1-\eta}$ supplied by \texttt{lemma\_Q\_p\_bound} \\
Lemma~\ref{lem:chernoff} &
\texttt{lemma\_chernoff\_inequality} and \texttt{lemma\_chernoff\_binomial}, via \texttt{lemma\_exp\_sum\_X\_p} / \texttt{lemma\_markov\_exp} \\
Lemma~\ref{lem:badcarry-count} &
\texttt{lemma\_bad\_carries\_bound} and the subsequent summation lemmas \texttt{lemma\_sum\_bad\_carries\_le\_bound}, \texttt{lemma\_sum\_bad\_carries\_small}\\
Lemma~\ref{lem:badspike-count} &
\texttt{lemma\_spike\_count\_bound} and \texttt{lemma\_sum\_bad\_spikes\_le\_bound}\\
Lemma~\ref{lem:bad-small} &
Proved inside \texttt{lemma\_good\_m\_exists\_any\_c} via various auxiliary lemmas \\
Lemma~\ref{lem:good-m} &
\texttt{lemma\_good\_m\_exists\_any\_c}, then used in \texttt{erdos\_728\_fc}\\
\hline
\end{tabular}
\end{center}

\section*{Appendix: the story of this proof}

While the mathematical and literature contents of the writeup is already complete, this proof has a significance of being the first recognized Erd\H{o}s problem (a problem proposed by Paul Erd\H{o}s and collected on the Erd\H{o}s problems website \url{https://www.erdosproblems.com}) that was solved autonomously by an AI system, with no prior literature found as of the date of this writing.

As this is a milestone, it is of interest to ask about the nature of the proof, e.g. how complex/novel it is compared to other mathematical results. Unfortunately (or fortunately) the resulting proof is in formalized form in Lean. While this boosts confidence in correctness, it may not be readily digestible.

This writeup attempts to address this issue. The author has received substantial assistance from ChatGPT in writing this manuscript. Indeed, a large fraction of the words were penned by ChatGPT. However, the author does not take this to mean we compromise on correctness; in contrast, everything has been manually checked up to roughly the level that the author himself would be confident in were he to write it himself and send it off as a journal submission! Of course, mistakes are still possible and any correction is welcome.

It is of interest to record the story of this proof as follows. The ``raw'' story may be read in the Erd\H{o}s problems website thread: \url{https://www.erdosproblems.com/728}.

On Jan 4, 2026, Kevin Barreto announced that he had a proof from the AI system: Aristotle by Harmonic. This is a Lean system, and the input to Aristotle was based on an informal argument from GPT-5.2 Pro. However there was one nuance here: at that time, the problem as stated on the website was vague, with the intended interpretation from Erd\H{o}s unclear. Barreto's version resolved \emph{one} version of the problem, but it wasn't clear whether this was the intended version.

Afterwards, forum participants commented on this vagueness; a consensus then emerged that this should be regarded as a \emph{partial result}, which did not yet fully resolve the problem.

But - on Jan 5, 2026, Barreto again asked GPT-5.2 Pro whether its argument could be upgraded to tackle the version that forum participants had now identified as \emph{the} problem. GPT-5.2 Pro responded in affirmative; Aristotle was then run based on this latest response and produced a formal proof on Jan 6, 2026.

There was a misunderstanding around this time that this second output had been based on a human's mathematical observation, which would render the result not autonomous by the AIs involved. Nevertheless, Barreto had clarified that this was a misunderstanding and a human observation was not provided in order to reach this result. With autonomy confirmed, the problem was now understood to be fully resolved by AI.

In particular, Terence Tao has vouched for this autonomous status which helped in cementing the consensus.

Afterwards, a forum participant who goes by KoishiChan attempted a literature review to locate any prior literature resolving this problem. KoishiChan has been highly successful in locating prior literature in the past for problems claimed to be solved by AI; however this time around an existing literature resolving this problem was not found. It is of course possible that such literature exists, as it is impossible to disprove the possibility, but we have not found such literature as of the date of this writing. Thus, the current status is that \emph{Erd\H{o}s Problem 728 was fully resolved autonomously by AI with no prior human literature resolving the problem found}.

On Jan 6, 2026, Boris Alexeev ran Aristotle again on the Aristotle proof in order to simplify it, producing a new proof. This is the Lean proof in \cite{AlexeevLean} which this present writeup is based on.

The author has then worked with ChatGPT in order to extract a human-readable proof from \cite{AlexeevLean} to make it accessible to a wider audience. Later, the presentation was iteratively improved and relevant literature included in order to make it a fuller presentation with maximum value to the mathematical community; this is the present writeup.

For reference, the conversation with ChatGPT that the author collaborated with in producing this writeup can be accessed here.\footnote{\url{https://chatgpt.com/share/696148ea-ffd4-8010-830f-7c04f3ab5cb9}} Note that while there are many drafts there, none of them is the current writeup, in which the author further applied significant rewrites in various places by himself.

\section*{Appendix: beyond this proof}

Shortly after this proof was completed, Erd\H{o}s problem \#729 was then solved on Jan 10, 2026, and \#401 on Jan 11, 2026, again autonomously by AI.  It was then subsequently noticed that these problems could be solved by modifications of the present argument as well.

Concurrently, forum participants asked Carl Pomerance, the author of \cite{Pomerance2015}, himself about the proof. On Jan 11, 2026, Pomerance replied that a result similar ours should be possible to obtain by modifying his argument in \cite{Pomerance2015} and using a lemma which was identical to our Lemma \ref{lem:W-upper}. Pomerance then wrote up his extension as a note \cite{Pomerance2026}.

In light of these developments, it is of interest to compare our results with \cite{Pomerance2026}, which turn out to be very similar.

In order to unify the solutions of \#728, \#729, and \#401 together, we must extract a general theorem out of the present argument. It is as follows.

\begin{theorem}\label{thm:general}
There exist absolute constants $c_1,c_2 > 0$ with the following properties. Consider the set $S$ of $m \in \bbN$ such that, for all $0 \le k \le \exp(c_1 \sqrt{\log m})$,
\[
\nu_p\left(\binom{2m}{m}\right) - \nu_p\left(\binom{m+k}{m}\right) \ge c_2 \frac{\log m}{\log p} \cdot [p \le 2k]
\]
for all primes $p$. Then $S$ has asymptotic density $1$.
\end{theorem}

\begin{proof}
If we follow the reductions as in the writeup, to get $m \in S$ it is enough to show
\[
\kappa_p(m) - V_p(m,k) \ge c_2 \frac{\log m}{\log p}
\]
for primes $p \le 2k$.

Consider $m \in [M, 2M]$ as before. Now we do not enforce $t(M) = O(\log \log M)$ but leave $t(M)$ unspecified for now. If we look at Lemma \ref{lem:weak-threshold} we can indeed produce a gap of the form
\[
c_2 \frac{\log m}{\log p}
\]
as required. The remaining thing to consider is that the union bound in Lemma \ref{lem:bad-small} must still produce $|\mathrm{Bad}(M)| = o(M)$. This works out, except that we must choose $t(M)$ to be order of $\sqrt{\log M}$. Now we look back at Lemma \ref{lem:weak-threshold} and this choice still works.
\end{proof}

To solve problem \#728 \cite{Bloom728}, it is enough to show that the left-hand side in Theorem \ref{thm:general} is nonnegative, which is immediate from the theorem.

To solve problem \#729 \cite{Bloom729}, this problem corresponds to the following. For every $c > 0$, taking $k = \lfloor c \log m \rfloor$, we want to produce a threshold $p_{\min}(c)$ such that for all primes $p \ge p_{\min}$, the denominator of $(2m)!/(m!(m+k)!)$ is not divisible by $p$. By Theorem \ref{thm:general}, we must show
\[
c_2 \frac{\log m}{\log p} \cdot [p \le 2k] - \nu_p(k!) \ge 0.
\]
If $p > 2k$ the left-hand side is zero. So assume $p \le 2k$. If we use $\nu_p(k!) \le k/(p-1)$, it is seen that this is true for $p$ large, depending only on $c$, as required.

To solve problem \#401 \cite{Bloom401}, this problem is just a more precise version of \#729. The same examples work. There is an extra condition that must be satisfied which is
\[
\nu_p(m!) + \nu_p((m+k)!) - \nu_p((2m)!) \le 2m,
\]
for small primes, depending on $c$.
But by Legendre's formula the left-hand side is $O(\log m)$, so this is easy to satisfy.

Finally, in \cite{Pomerance2015} it was shown that
\[
(m+1)(m+2)\cdots (m+k)\ \Big|\ \binom{2m}{m}
\]
where $m$ is in a set of asymptotic density $1$, and $k$ is fixed or grows slowly with $m$. In \cite{Pomerance2026}, this was explicitly worked out to be $k \le \eta \log m$ for any $\eta < 1/\log 4$. We obtain a slightly weaker result, with $k \le c \log m$ for some small constant $c > 0$. The key is again the reduction to
\[
c_2 \frac{\log m}{\log p} \cdot [p \le 2k] - \nu_p(k!) \ge 0
\]
as in problem \#729. Then we follow a similar reasoning. In \cite{Pomerance2026} it was also shown that $\binom{m+k}{m} \mid \binom{2m}{m}$ for $k \le \exp(.8 \sqrt{\log m})$, where $m$ is in a set of asymptotic density $1$. This is slightly stronger than Theorem \ref{thm:general}, as the constant was explicitly worked out.

The next Appendix derives an effective version of Theorem \ref{thm:general} which may be of interest. In particular, the optimality of the result is investigated.

\section*{Appendix: effective bounds}

The following result was derived in a conversation with ChatGPT.\footnote{\url{https://chatgpt.com/share/696aa0bd-c560-8010-ab97-0cb938065f33}} It was drafted by ChatGPT and edited and checked by the author. There is one key new idea, which is that to get the best constants, $\theta(p)$, which can be as low as $1/3$, must be replaced by $1/2$. This can be done by inspecting the Markov chain from carry propagation. The results needed are the following.

\begin{lemma}[Carry chain]\label{lem:carry-chain}
Let $p$ be a prime and $L\ge 1$. Let $m$ be uniform in $\{0,1,\dots,p^L-1\}$. Write $
m=\sum_{i=0}^{L-1} a_i p^i$, $a_i\in\{0,1,\dots,p-1\}$, and define carries $(C_i)_{i=0}^L\subset\{0,1\}$ by $C_0:=0$ and $C_{i+1}:=\mathbf{1}\{\,2a_i + C_i \ge p\,\}$. Set $S_L:=\sum_{i=1}^L C_i$.

For $\lambda\in\mathbb{R}$ define the tilted matrix $T_p(\lambda) := \big(P_p(u,v)e^{\lambda v}\big)_{u,v\in\{0,1\}}$,
where $P_p$ is the transition matrix of the carry chain $(C_i)$. Let $\rho_p(\lambda)$ be the Perron--Frobenius eigenvalue of $T_p(\lambda)$.

\begin{enumerate}
\item We have
\[
\mathbb{E}\big[e^{\lambda S_L}\big]=e_0^\top T_p(\lambda)^L \mathbf{1},
\qquad e_0=(1,0)^\top,\ \mathbf{1}=(1,1)^\top.
\]
Moreover, there is a constant $C_p(\lambda) > 0$, depending only on $p,\lambda$ and not on $L$, such that
\[
\mathbb{E}\big[e^{\lambda S_L}\big]\ \le\ C_p(\lambda)\,\rho_p(\lambda)^L.
\]
For each fixed $\lambda$ there exist $p_1(\lambda)$ and $C(\lambda)$ such that $\sup_{p\ge p_1(\lambda)} C_p(\lambda) =: C(\lambda)<\infty$.

\item For any $s\in(0,1)$ and any $\lambda<0$,
\[
\mathbb{P}(S_L\le sL)\ \le\ C_p(\lambda)\,\exp\!\big(L(\log\rho_p(\lambda)-\lambda s)\big).
\]

\item Fix $\delta\in(0,1)$ and $\varepsilon>0$, set $s:=(1-\delta)/2$ and
\[
I(\delta):=D\!\left(\frac{1-\delta}{2}\,\Big\|\,\frac12\right)
=\frac12\Big((1-\delta)\log(1-\delta)+(1+\delta)\log(1+\delta)\Big).
\]
Then there exist $p_0=p_0(\delta,\varepsilon)$ and $C(\delta) > 0$ such that for all primes $p\ge p_0$ and all $L\ge 1$,
\[
\mathbb{P}(S_L\le sL)\ \le\ C(\delta)\,\exp\!\big(-(I(\delta)-\varepsilon)L\big).
\]

\item For each fixed prime $p$ and each $\delta\in(0,1)$ there exist constants
$\gamma_p(\delta)>0$ and $C_p(\delta) > 0$ such that for all $L\ge 1$,
\[
\mathbb{P}(S_L\le sL)\ \le\ C_p(\delta)\,e^{-\gamma_p(\delta)L}.
\]
\end{enumerate}

Note that what the constant $C$ depends on depends on context.
\end{lemma}

\begin{proof}
Conditioning on $C_i$ and using that $a_i$ is uniform gives the two-state Markov chain. For $p\ge 3$,
\[
P_p=
\begin{pmatrix}
\frac12+\frac{1}{2p} & \frac12-\frac{1}{2p}\\[4pt]
\frac12-\frac{1}{2p} & \frac12+\frac{1}{2p}
\end{pmatrix},
\]
and for $p=2$ the carries are independent so both rows equal $(1/2,1/2)$.

(1) The identity $\mathbb{E}[e^{\lambda S_L}]=e_0^\top T_p(\lambda)^L\mathbf{1}$ follows by iterating the recursion. Let $v_{p,\lambda}>0$ be a right Perron--Frobenius eigenvector: $T_p(\lambda)v_{p,\lambda}=\rho_p(\lambda)v_{p,\lambda}$.
Normalize $v_{p,\lambda}$ so that $\min(v_{p,\lambda,0},v_{p,\lambda,1})=1$ and let
$R_{p,\lambda}:=\max(v_{p,\lambda,0},v_{p,\lambda,1})$. Then, entrywise, $\mathbf{1}\le R_{p,\lambda}v_{p,\lambda}$, so
\[
T_p(\lambda)^L\mathbf{1}\ \le\ R_{p,\lambda}T_p(\lambda)^L v_{p,\lambda}
=R_{p,\lambda}\rho_p(\lambda)^L v_{p,\lambda}\ \le\ R_{p,\lambda}^2\rho_p(\lambda)^L\mathbf{1}.
\]
Hence
\[
\mathbb{E}[e^{\lambda S_L}] \le R_{p,\lambda}^2\rho_p(\lambda)^L.
\]
Thus we can take $C_p(\lambda):=R_{p,\lambda}^2$.

For fixed $\lambda$, the matrices $T_p(\lambda)$ converge entrywise to
\[
T_\infty(\lambda)=\frac12\begin{pmatrix}1 & e^\lambda\\ 1 & e^\lambda\end{pmatrix}\qquad(p\to\infty),
\]
and $T_\infty(\lambda)$ has strictly positive entries. By continuity of the Perron--Frobenius eigenvector, $R_{p,\lambda}$ remains bounded for all sufficiently large $p$.
Therefore $\sup_{p\ge p_1(\lambda)} C_p(\lambda) <\infty$.

(2) For $\lambda<0$, on the event $\{S_L\le sL\}$ we have $e^{\lambda S_L}\ge e^{\lambda sL}$, so Markov's inequality gives
\[
\mathbb{P}(S_L\le sL)\le e^{-\lambda sL}\,\mathbb{E}[e^{\lambda S_L}]
\le C_p(\lambda)\exp\!\big(L(\log\rho_p(\lambda)-\lambda s)\big).
\]

(3) Fix $\delta,\varepsilon$. Choose the Bernoulli$(1/2)$ optimal tilt
\[
\lambda^*:=\log\Big(\frac{s}{1-s}\Big)=\log\Big(\frac{1-\delta}{1+\delta}\Big)<0.
\]
For the limit matrix $T_\infty(\lambda)$, the Perron--Frobenius eigenvalue is
\[
\rho_\infty(\lambda)=\frac{1+e^\lambda}{2}.
\]
A direct computation gives $\lambda^* s-\log\rho_\infty(\lambda^*)=I(\delta)$.
Since $\rho_p(\lambda^*)\to\rho_\infty(\lambda^*)$, we have
$\log\rho_p(\lambda^*)\le \log\rho_\infty(\lambda^*)+\varepsilon$ for all primes $p\ge p_0(\delta,\varepsilon)$.
By (1) applied to $\lambda=\lambda^*$, enlarging $p_0$ if needed, we have
$\sup_{p\ge p_0} C_p(\lambda^*)\le C(\delta)$ for some finite $C(\delta)$.
Now (2) gives, for $p\ge p_0$,
\[
\mathbb{P}(S_L\le sL)\le C(\delta)\exp\!\big(L(\log\rho_p(\lambda^*)-\lambda^* s)\big)
\le C(\delta)\exp\!\big(-(I(\delta)-\varepsilon)L\big).
\]

(4) Fix $p,\delta$.
Define $\Lambda_p(\lambda):=\log\rho_p(\lambda)$.
Then $\Lambda_p$ is convex, $\Lambda_p(0)=0$, and $\Lambda_p'(0)=1/2$.
So $\sup_{\lambda<0}\{\lambda s-\Lambda_p(\lambda)\}>0$. Choose $\lambda_0<0$ such that
$\lambda_0 s-\Lambda_p(\lambda_0)=: \gamma_p(\delta)>0$ and set $C_p(\delta):=C_p(\lambda_0)$.
By (2) with $\lambda=\lambda_0$
\[
\mathbb{P}(S_L\le sL)\le C_p(\delta)\exp\!\big(L(\Lambda_p(\lambda_0)-\lambda_0 s)\big)
= C_p(\delta)e^{-\gamma_p(\delta)L}.
\]
\end{proof}

We now derive the effective bounds.

\begin{theorem}[General gap with a wider prime range]\label{thm:general2}
Let
\[
c_*:=\sqrt{\log 2},
\qquad
I(\delta):=\frac12\Big((1-\delta)\log(1-\delta)+(1+\delta)\log(1+\delta)\Big)
\qquad(0<\delta<1).
\]
Fix constants $0<c<c_p<c_*$ and choose $\delta\in(0,1)$ such that $c_p^2<I(\delta)$.
Define
\[
K(m):=\Big\lfloor \exp\!\big(c\sqrt{\log m}\big)\Big\rfloor,
\qquad
P(m):=\Big\lfloor \exp\!\big(c_p\sqrt{\log m}\big)\Big\rfloor.
\]
Let $S$ be the set of $m\in\bbN$ such that, for all integers $k$ with $0\le k\le K(m)$,
\[
\nu_p\!\left(\binom{2m}{m}\right)-\nu_p\!\left(\binom{m+k}{m}\right)
\ \ge\
\frac{1-\delta}{2}\,\frac{\log m}{\log p}
\cdot [\,p\le P(m)\,]
\]
for all primes $p$. Then $S$ has asymptotic density $1$.
\end{theorem}

\begin{proof}
We work on a scale $m\in[M,2M]$ and show that, for each large $M$, the desired inequalities hold for all but $o(M)$ integers
$m\in[M,2M]\cap\bbN$. The conclusion then follows.

Since $I$ is continuous and strictly increasing, we may choose $\delta_0\in(0,\delta)$ with $c_p^2<I(\delta_0)$.
We choose $\eta\in(0,1)$ and $\varepsilon>0$ so small that
\begin{equation}\label{eq:cp-margin}
c_p^2 < (1-\eta)\bigl(I(\delta_0)-\varepsilon\bigr),
\qquad\text{and}\qquad
(1-\eta)(1-\delta_0) > 1-\delta.
\end{equation}
Set
\[
t(M):=\left\lceil \log\log M\right\rceil,
\]
so $t(M)\to\infty$ and $t(M)=o(\sqrt{\log M})$.

Let $K := K(2M)$, $P := P(2M)$. For each prime $p\le P$, define the digit depth
\[
L_p:=\left\lfloor \frac{(1-\eta)\log M}{\log p}\right\rfloor
\qquad\text{and}\qquad
\mu_p:=\frac{L_p}{2}.
\]
Write the base-$p$ expansion
$m \equiv \sum_{i=0}^{L_p-1} a_i p^i \pmod{p^{L_p}}$, $a_i\in\{0,1,\dots,p-1\}$.
Let $C_0:=0$ and define carries $C_{i+1}\in\{0,1\}$ by $C_{i+1}:=\mathbf{1}\{\,2a_i + C_i \ge p\,\}$.
Define the truncated carry count
\[
Y_p(m):=\sum_{i=1}^{L_p} C_i.
\]

For each prime $p\le P$, define
\[
\mathrm{BadCarry}_p(M)
:=\bigl\{m\in[M,2M]: Y_p(m) < (1-\delta_0)\mu_p \bigr\},
\]
and
\[
\mathrm{BadSpike}_p(M)
:=\bigl\{m\in[M,2M]: V_p(m,K) \ge J_p+t(M)\bigr\},
\qquad
J_p:=\lfloor \log_p K\rfloor.
\]
Let
\[
\mathrm{Bad}(M):=\bigcup_{p\le P}\Big(\mathrm{BadCarry}_p(M)\cup \mathrm{BadSpike}_p(M)\Big).
\]
We claim $|\mathrm{Bad}(M)|=o(M)$.

The event $m\in\mathrm{BadCarry}_p(M)$ depends only on $m \bmod p^{L_p}$.
Among residues $r \bmod p^{L_p}$, the digits are i.i.d.\ uniform and the carry variables $(C_i)$ form the Markov chain of Lemma~\ref{lem:carry-chain},
so $Y_p(r)$ has the same law as $S_{L_p}$ in that lemma.
By Lemma~\ref{lem:carry-chain}(3) with $\delta=\delta_0$, there are $p_0=p_0(\delta_0,\varepsilon)$ and $C(\delta_0)$ such that for every prime $p\ge p_0$,
\[
\mathbb{P}\!\left(Y_p(m) < (1-\delta_0)\frac{L_p}{2}\right)
\ \le\ C(\delta_0)\exp\!\big(-(I(\delta_0)-\varepsilon)L_p\big).
\]
For the finitely many primes $p<p_0$, Lemma~\ref{lem:carry-chain}(4) gives
\[
\mathbb{P}\!\left(Y_p(m) < (1-\delta_0)\frac{L_p}{2}\right)\ \le\ C_p(\delta_0)e^{-\gamma_p(\delta_0)L_p}.
\]
The total contribution of these is $o(M)$ and may be absorbed.
For $p \ge p_0$, the same residue-counting argument as Lemma~\ref{lem:badcarry-count} yields
\[
|\mathrm{BadCarry}_p(M)|
\ \le\
(M+1)\,C(\delta_0)\exp\!\big(-(I(\delta_0)-\varepsilon)L_p\big)+2p^{L_p}.
\]
Now sum over $p\le P$.
Since $p^{L_p}\le M^{1-\eta}$, the boundary terms contribute $o(M)$.
For the exponential terms,
\[
L_p \ \ge\ \frac{(1-\eta)\log M}{\log P}-1
\ =\ \frac{1-\eta}{c_p}\sqrt{\log M}+O(1).
\]
So the total contribution of exponential terms is
\[
\le\
C(\delta_0) M\cdot \pi(P)\cdot \exp\!\Big(-\frac{(1-\eta)(I(\delta_0)-\varepsilon)}{c_p}\sqrt{\log M}+O(1)\Big)
=o(M)
\]
by \eqref{eq:cp-margin}.

For spikes, the proof of Lemma~\ref{lem:badspike-count} gives, for each prime $p$,
\[
|\mathrm{BadSpike}_p(M)| \ \le\ (M+1)\,p^{1-t(M)}+2K.
\]
Summing over $p\le P$ gives
\[
\sum_{p\le P}|\mathrm{BadSpike}_p(M)|
\ \le\
(M+1)\sum_{p\le P}p^{1-t(M)}+2K\pi(P).
\]
We bound $\sum_{p\le P}p^{1-t(M)}\le \sum_{n\ge2}n^{1-t(M)}=\zeta(t(M)-1)-1=o(1)$ since $t(M)\to\infty$, where $\zeta$ is the Riemann zeta function.
Also $K\pi(P)=o(M)$.

Combining everything, we have shown $|\mathrm{Bad}(M)|=o(M)$.

Let $m\in [M,2M] \setminus \mathrm{Bad}(M)$, and let $0\le k\le K(m)$ and prime $p\le P(m)$. We have $Y_p(m)\ge (1-\delta_0)\mu_p$ and $V_p(m,K)<J_p+t(M)$.
By Kummer's theorem, $\kappa_p(m)\ge Y_p(m)$.
By Lemmas~\ref{lem:val-reduction} and \ref{lem:W-upper}, $\nu_p\big(\binom{m+k}{m}\big)\le V_p(m,k)\le V_p(m,K)$.
So
\begin{equation}\label{eq:gap-d0}
\nu_p\!\left(\binom{2m}{m}\right)-\nu_p\!\left(\binom{m+k}{m}\right)
\ \ge\
(1-\delta_0)\mu_p - (J_p+t(M)).
\end{equation}
Note
\[
(1-\delta_0)\mu_p
\ \ge\
\frac{1-\delta_0}{2} \left(\frac{(1-\eta)\log M}{\log p}-1\right).
\]
Also $J_p\ll \frac{\sqrt{\log M}}{\log p}$ and $t(M)=o(\sqrt{\log M})$.
As there is a difference between the coefficient $(1-\delta_0)(1-\eta)$ and $(1-\delta)$ from \eqref{eq:cp-margin}, the $J_p+t(M)+O(1)$ subtraction can be absorbed.
Thus for all large $M$ and all primes $p\le P(m)$,
\[
\nu_p\!\left(\binom{2m}{m}\right)-\nu_p\!\left(\binom{m+k}{m}\right)
\ \ge\
\frac{1-\delta}{2}\,\frac{\log m}{\log p}.
\]
Finally, for primes $p>P(m)$, we have $p \ge 2k$ for large $m$,
so Lemma~\ref{lem:large-primes} implies
\[
\nu_p\!\left(\binom{2m}{m}\right)-\nu_p\!\left(\binom{m+k}{m}\right)\ge 0.
\]
\end{proof}

\begin{remark}
As $I(1) = \log 2$, $c_* = \sqrt{I(1)}$ is as far as the argument can go. As a consequence, we obtain the following corollaries with constants matching Pomerance's note \cite{Pomerance2026}.
\end{remark}

\begin{corollary}\label{cor:po-thm1-2}
Let $c < \sqrt{\log 2} \approx 0.833$. Let $S$ be the set of $m \in \bbN$ such that
\[
\binom{m+k}{m}\ \Big|\ \binom{2m}{m}
\]
for all $k \le \exp(c \sqrt{\log m})$. Then $S$ has asymptotic density $1$.
\end{corollary}

\begin{proof}
Immediate from Theorem \ref{thm:general2}.
\end{proof}

\begin{corollary}\label{cor:po-thm1-1}
Let $c < 1/(2 \log 2) \approx 0.721$. Let $S$ be the set of $m \in \bbN$ such that
\[
(m+1)(m+2)\cdots (m+k)\ \Big|\ \binom{2m}{m}
\]
for all $k \le c \log m$. Then $S$ has asymptotic density $1$.
\end{corollary}

\begin{proof}
As in the previous Appendix, we are reduced to
\[
\frac{1-\delta}{2}\,\frac{\log m}{\log p}
\cdot [\,p\le P(m)\,] - \nu_p(k!) \ge 0.
\]
If $p>k$ this is obvious, so assume $p\le k$. Now $\nu_p(k!) \le k/(p-1) \le c \log m/(p-1)$. In Theorem \ref{thm:general2}, we can take $\delta > 0$ very small, so we just need
\[
\frac{1}{2 \log p} - \frac{c}{p-1} > 0.
\]
This is strictest at $p=2$, corresponding to the assumption.
\end{proof}

Finally, we investigate the optimality of these constants with ChatGPT.\footnote{\url{https://chatgpt.com/share/696d141d-5b98-8010-935d-9e6f12caae5d}} We prove Corollary \ref{cor:po-thm1-1} is optimal. Moreover, Corollary \ref{cor:po-thm1-2} is plausibly optimal by a heuristic but we are not able to prove this.

\begin{proposition}[Corollary \ref{cor:po-thm1-1} is optimal]
\label{prop:cor2-sharp}
Let $c>1/(2\log 2)$ and $k(m):=\lfloor c\log m\rfloor$.
Then for a set of $m \in \bbN$ with asymptotic density~$1$,
\[
(m+1)(m+2)\cdots(m+k(m))\ \nmid\ \binom{2m}{m}.
\]
\end{proposition}

\begin{proof}
We have that $(m+1)\cdots(m+k)\mid\binom{2m}{m}$ implies $k!\mid\binom{2m}{m}$. So it suffices to find a prime $p$ for which $\nu_p(k!)>\nu_p\binom{2m}{m}$. In fact, take $p=2$. By Legendre's formula $\nu_2(n!)=n-s_2(n)$, where $s_2(n)$ is the sum of binary digits of $n$. Hence
\[
\nu_2\binom{2m}{m}
=2s_2(m)-s_2(2m)=s_2(m).
\]
Now $\nu_2(k!)=k-s_2(k)= (1-o(1))k$.
It is a standard fact that for a fixed $\varepsilon>0$,
\[
 s_2(m)\le \Big(\tfrac12+\varepsilon\Big)\log_2 m
\qquad\text{for a set of $m \in \bbN$ with asymptotic density $1$.}
\]
The conclusion follows.
\end{proof}

\smallskip
For Corollary \ref{cor:po-thm1-2}, the value $\sqrt{\log 2}$ is plausibly optimal. We record the following heuristic argument. Let $c>\sqrt{\log 2}$ and put $K(m):=\big\lfloor e^{c\sqrt{\log m}}\big\rfloor$.
To show $m\notin S$ it suffices to find a prime $p$ and some $k\le K(m)$ such that
\[
\nu_p\binom{m+k}{m}\ge 1\quad\text{but}\quad \kappa_p(m)=0.
\]
By Kummer, $\kappa_p(m)=0$ is equivalent to no carries when doubling $m$ in base $p$, i.e. all base-$p$ digits of $m$ are in $\{0,1,\dots,\lfloor(p-1)/2\rfloor\}$.

We want to take primes in a window $p\in(K(m),(1+\delta)K(m))$ for a fixed small $\delta\in(0,1)$ (so $p<2K(m)$).
Then $K(m)<p$ and, for a fixed $k:=K(m)$,
$p\mid \binom{m+k}{m}$ iff $ m\bmod p\in\{p-k,\dots,p-1\}$.

So we need the conditions
\begin{enumerate}
\item[(i)] all base-$p$ digits of $m$ are $\le (p-1)/2$;
\item[(ii)] the least significant digit is in $\{p-k,\dots,\lfloor(p-1)/2\rfloor\}$.
\end{enumerate}
Condition (ii) requires $p\le 2k-1$, and in the window $p\in(K,(1+\delta)K)$ it holds with a fixed positive probability conditional on (i).

Fix a large $M$ and consider $m\in[M,2M]$, so $K(m)=K(M)(1+o(1))$.
Let $K:=K(2M)$ and fix a prime $p\in(K,(1+\delta)K)$. Let $L\approx \log_p M$ be the number of base-$p$ digits of $m$. Then
\[
\mathbb{P}(\kappa_p(m)=0)\ \approx\ \Big(\frac{1}{2}\Big)^{L}\
\ \approx\ \exp\Big(-\frac{\log 2}{c}\sqrt{\log M}\Big).
\]
Note multiplying by the constant conditional probability from (ii) does not change the scale. The number of primes in $(K,(1+\delta)K)$ is
\[
\pi((1+\delta)K)-\pi(K)\ \asymp\ \frac{\delta K}{\log K}
\ \asymp\ \frac{1}{O(\sqrt{\log M})}\,\exp\big(c\sqrt{\log M}\big).
\]
So the expected number of primes satisfying the obstruction conditions is heuristically
\[
\asymp\ \frac{1}{O(\sqrt{\log M})}
\exp\Big(\Big(c-\frac{\log 2}{c}\Big)\sqrt{\log M}\Big).
\]
This diverges as $M\to\infty$ precisely when $c>\sqrt{\log 2}$.


\begin{thebibliography}{9}

\bibitem{AlexeevLean}
B.~Alexeev,
\emph{Lean formalization of Erd\H{o}s Problem \#728 (\texttt{Erdos728b.lean})},
in the repository \emph{lean-proofs},
\url{https://github.com/plby/lean-proofs/blob/main/src/v4.24.0/ErdosProblems/Erdos728b.lean},
accessed 2026-01-27.

\bibitem{Bloom401}
T.~F.~Bloom,
\emph{Erd\H{o}s Problem \#401},
Erd\H{o}s Problems website,
\url{https://www.erdosproblems.com/401},
accessed 2026-01-27.

\bibitem{Bloom728}
T.~F.~Bloom,
\emph{Erd\H{o}s Problem \#728},
Erd\H{o}s Problems website,
\url{https://www.erdosproblems.com/728},
accessed 2026-01-27.

\bibitem{Bloom729}
T.~F.~Bloom,
\emph{Erd\H{o}s Problem \#729},
Erd\H{o}s Problems website,
\url{https://www.erdosproblems.com/729},
accessed 2026-01-27.

\bibitem{BloomCroot2025}
T.~F.~Bloom and E.~Croot,
\emph{Integers with small digits in multiple bases},
arXiv:2509.02835 (2025).

\bibitem{CrootMousaviSchmidt2024}
E.~Croot, H.~Mousavi, and M.~Schmidt,
\emph{On a conjecture of Graham on the $p$-divisibility of central binomial coefficients},
Mathematika \textbf{70} (2024), no.~3, Article ID e12249.

\bibitem{Erdos557}
P.~Erd\H{o}s,
\emph{Aufgabe 557},
Elemente der Mathematik \textbf{23} (1968), 111--113.

\bibitem{EGRS1975}
P.~Erd\H{o}s, R.~L.~Graham, I.~Z.~Ruzsa, and E.~G.~Straus,
\emph{On the prime factors of $\binom{2n}{n}$},
Mathematics of Computation \textbf{29} (1975), no.~129, 83--92.

\bibitem{FordKonyagin2021}
K.~Ford and S.~Konyagin,
\emph{Divisibility of the central binomial coefficient $\binom{2n}{n}$},
Transactions of the American Mathematical Society \textbf{374} (2021), no.~2, 923--953.

\bibitem{Pomerance2015}
C.~Pomerance,
\emph{Divisors of the middle binomial coefficient},
American Mathematical Monthly \textbf{122} (2015), no.~7, 636--644.

\bibitem{Pomerance2026}
C.~Pomerance,
\emph{A remark on the middle binomial coefficient},
preprint,
\url{https://math.dartmouth.edu/~carlp/binomrev2.pdf},
accessed 2026-01-27.

\end{thebibliography}
\end{document}